\documentclass{amsart}
\usepackage{amsmath,amssymb,amsthm}

\begin{document}





\newtheorem{thm}{Theorem}[section]
\newtheorem{prop}[thm]{Proposition}
\newtheorem{cor}[thm]{Corollary}
\newtheorem{defn}{Definition}[section]
\newtheorem{rem}{Remark}[section]


\newcommand{\X}{\mathfrak{X}}
\newcommand{\B}{\mathcal{B}}
\newcommand{\s}{\mathfrak{S}}
\newcommand{\g}{\mathfrak{g}}
\newcommand{\W}{\mathcal{W}}
\newcommand{\T}{\mathcal{T}}
\newcommand{\Lgr}{\mathrm{L}}
\newcommand{\dd}{\mathrm{d}}
\newcommand{\n}{\nabla}
\newcommand{\nn}{\nabla'}
\newcommand{\lm}{\lambda}
\newcommand{\ta}{\theta}
\newcommand{\pd}{\partial}
\newcommand{\LL}{\mathcal{L}}

\newcommand{\grad}{\mathrm{grad}}
\newcommand{\End}{\mathrm{End}}
\newcommand{\im}{\mathrm{Im}}
\newcommand{\id}{\mathrm{id}}

\newcommand{\ie}{i.e. }
\newfont{\w}{msbm9 scaled\magstep1}
\def\R{\mbox{\w R}}
\newcommand{\norm}[1]{\left\Vert#1\right\Vert ^2}
\newcommand{\nnorm}[1]{\left\Vert#1\right\Vert ^{*2}}
\newcommand{\nN}{\norm{N}}
\newcommand{\nP}{\norm{\nabla P}}
\newcommand{\nnP}{\nnorm{\nabla P}}
\newcommand{\tr}{{\rm tr}}

\newcommand{\thmref}[1]{Theorem~\ref{#1}}
\newcommand{\propref}[1]{Proposition~\ref{#1}}
\newcommand{\secref}[1]{\S\ref{#1}}
\newcommand{\dfnref}[1]{Definition~\ref{#1}}



\title[A connection on Riemannian product manifolds]
{A natural connection on a basic class of Riemannian product
manifolds}

\author{Dobrinka Gribacheva}

\maketitle

\begin{abstract}
A Riemannian manifold $M$ with an integrable almost product
structure $P$ is called a Riemannian product manifold. Our
investigations are on the manifolds $(M,P,g)$ of the largest class
of Riemannian product manifolds, which is closed with respect to
the group of conformal transformations of the metric $g$. This
class is an analogue of the class of locally conformal K\"ahler
manifolds in almost Hermitian geometry. In the present paper we
study a natural connection $D$ on $(M,P,g)$ (\ie $DP=Dg=0$). We
find necessary and sufficient conditions the curvature tensor of
$D$ to have properties similar to the K\"ahler tensor in Hermitian
geometry. We pay attention to the case when $D$ has a parallel
torsion. We establish that the Weyl tensors for the connection $D$
and the Levi-Civita connection coincide as well as the invariance
of the curvature tensor of $D$ with respect to the usual conformal
transformation. We consider the case when $D$ is a flat
connection. We construct an example of the considered manifold by
a Lie group where $D$ is
a flat connection with non-parallel torsion.
\end{abstract}




\section*{Introduction}

The systematic development of the theory of Riemannian almost
product manifolds, \ie Riemannian manifolds with almost product
structure, was started by K. Yano in \cite{Ya}. The geometry of a
Riemannian almost product manifold $(M,P,g)$ is a geometry of the
Riemannian metric $g$ and the almost product structure $P$. There
are important in this geometry the linear connections with respect
to which the parallel transport determines an isomorphism of the
tangent spaces of the manifold $M$ with structures $g$ and $P$.
Such connections are called natural in \cite{Mi}. They are an
analogue of the Hetmitian connections in almost Hermitian
geometry.

If the almost product structure $P$ is integrable, \ie the torsion
tensor of $P$ (the Nijenhuis tensor) is zero, then the manifold
$(M,P,g)$ is called a Riemannian product manifold.
In the present work we study a natural connection $D$  on a manifold  $(M,P,g)$
belonging to the largest class of Riemannian product manifolds,
which is closed with respect to the group of the conformal transformations
of the metric $g$. This is the class $\W_1$ of the classification in \cite{Sta-Gri}.

The present paper is organized as follows.
In Section~\ref{sec1} we give some necessary facts about the
Riemannian almost product manifolds and the natural connections on
them.
In Section~\ref{sec2} we consider a natural connection $D$ from
the 2-parametric family of all natural connections on
$(M,P,g)\in\W_1$ obtained in \cite{Dobr}. We find a relation
between the curvature tensors of the Levi-Civita connection $\n$
and the considered connection $D$. As a corollary we obtain a
relation between the Ricci tensors and between the scalar
curvatures for $\n$ and $D$.
In Section~\ref{sec3} we find necessary and sufficient conditions
for a connection $D$ whose curvature tensor is a Riemannian
$P$-tensor. The notion of a Riemannian $P$-tensor is an analogue
of the notion of a K\"ahler tensor in Hermitian geometry.
In Section~\ref{sec4} we consider the case of a connection $D$
with parallel torsion.
In Section~\ref{sec5} we establish that the Weyl tensors for $\n$
and $D$ coincide.
In Section~\ref{sec6} we prove that the curvature tensor of $D$ is
invariant with respect to the usual conformal transformation of
the metric $g$.
In Section~\ref{sec7} we consider the case of flat connection $D$.
In Section~\ref{sec8} we construct an example of the considered
manifold by a Lie group, where $D$ is a flat connection with
non-parallel torsion.


\section{Preliminaries}\label{sec1}

Let $(M,P,g)$ be a \emph{Riemannian almost product manifold},
\ie{} a differentiable manifold $M$ with a tensor field $P$ of
type $(1,1)$ and a Riemannian metric $g$ such that
\begin{equation*}\label{2.1}
    P^2x=x,\quad g(Px,Py)=g(x,y)
\end{equation*}
for arbitrary $x$, $y$ of the algebra $\X(M)$ of the smooth vector
fields on $M$. Obviously $g(Px,y)=g(x,Py)$.

Further $x,y,z,w$ will stand for arbitrary elements of $\X(M)$ or
vectors in the tangent space $T_pM$ at $p\in M$.

In this work we consider Riemannian almost product manifolds with
$\tr{P}=0$. In this case $(M,P,g)$ is an even-dimensional
manifold. Let $\dim{M}$ be $2n$. Then the \emph{associated metric}
$\tilde{g}$ of $g$, determined by $\tilde{g}(x,y)=g(x,Py)$, is an
indefinite metric of signature $(n,n)$. We suppose that $\dim
M\geq 4$.

In \cite{Nav} A.M.~Naveira gives a classification of Riemannian
almost product manifolds with respect to the tensor $F$ of type
(0,3), defined by
\begin{equation}\label{2.2}
F(x,y,z)=g\left(\left(\nabla_x P\right)y,z\right),
\end{equation}
where $\n$ is the Levi-Civita connection of $g$. The tensor $F$
has the following properties:
\begin{equation*}\label{2.3}
\begin{array}{c}
    F(x,y,z)=F(x,z,y)=-F(x,Py,Pz),\\[4pt]
    F(x,y,Pz)=-F(x,Py,z).
\end{array}
\end{equation*}

Using the Naveira classification, in \cite{Sta-Gri} M.~Staikova
and K.~Gribachev give a classification of Riemannian almost
product manifolds $(M,P,g)$ with $\tr P=0$. The basic classes of
the classification in \cite{Sta-Gri} are $\W_1$, $\W_2$ and
$\W_3$. Their intersection is the class $\W_0$ of the
\emph{Riemannian $P$-manifolds}, determined by the condition $F=0$
or equivalently $\n P=0$ \cite{Sta87}. This class is the analogue
of the class of K\"ahler manifolds in the geometry of almost
Hermitian manifolds.

A Riemannian almost product manifold $(M,P,g)$ is a
\emph{Riemannian product manifold} if it has a local product
structure. This means that the almost product structure $P$ is
integrable, \ie the Nijenhuis tensor $N$ determined by
\begin{equation*}\label{2.3'}
    N(x,y)=[Px,Py]+[x,y]-P[Px,y]-P[x,Py]
\end{equation*}
is zero. The Riemannian product manifolds form the class
$\W_1\oplus\W_2$ from the classification in \cite{Sta-Gri}. This
class is an analogue of Hermitian manifolds in almost Hermitian
geometry.

The class $\W_1$ from the classification in \cite{Sta-Gri} consist
of the manifolds which are locally conformal equivalent to
Riemannian $P$-manifolds. This class plays a similar role of the
role of the class of the conformal K\"ahler manifolds in almost
Hermitian geometry \cite{Gray-Her}. The characteristic condition
for the class $\W_1$ is the following
\begin{equation}\label{2.4}
\begin{split}
& F(x,y,z)=\frac{1}{2n}\big\{ g(x,y)\ta (z)+g(x,z)\ta (y)-g(x,Py)\ta (Pz) \big.\\
& \phantom{F(x,y,z)=\frac{1}{2n}\big\{g(x,y)\ta (z)+g(x,z)\ta (y)
\big.} -g(x,Pz)\ta (Py)\big\},
\end{split}
\end{equation}
where $\ta$ is the associated Lee 1-form for $F$ determined by
\begin{equation}\label{2.4'}
\ta(x)=g^{ij}F(e_i,e_j,x).
\end{equation}
Here and further $g^{ij}$ will stand for the components of the
inverse matrix of $g$ with respect to a basis $\{e_i\}$ of $T_pM$
at $p\in M$.

The curvature tensor $R$ of $\n$ is determined by
$R(x,y)z=\nabla_x \nabla_y z - \nabla_y \nabla_x z -
    \nabla_{[x,y]}z$ and the corresponding tensor of type (0,4) is defined as
follows $R(x,y,z,w)=g(R(x,y)z,w)$. We denote the Ricci tensor and
the scalar curvature for $\n$ by $\rho$ and $\tau$, respectively,
\ie $\rho(y,z)=g^{ij}R(e_i,y,z,e_j)$ and
$\tau=g^{ij}\rho(e_i,e_j)$.

In \cite{Mek1}, a tensor $L$ of type (0,4) with pro\-per\-ties%
\begin{equation}\label{2.5}
L(x,y,z,w)=-L(y,x,z,w)=-L(x,y,w,z),
\end{equation}
\begin{equation}\label{2.6}
\mathop{\s} \limits_{x,y,z} L(x,y,z,w)=0,
\end{equation}
\begin{equation}\label{2.7}
L(x,y,Pz,Pw)=L(x,y,z,w),
\end{equation}
is called a \emph{Riemannian $P$-tensor}. This notion is an
analogue of the notion of a K\"ahler tensor in Hermitian geometry.

The linear connections in our investigations have a torsion.

Let $\nn$ be a linear connection with a tensor $Q$ of the
transformation $\n \rightarrow\nn$ and a torsion $T$, \ie{}
\begin{equation}\label{2.8}
\nn_x y=\n_x y+Q(x,y),\quad T(x,y)=\nn_x y-\nn_y x-[x,y].
\end{equation}
The corresponding (0,3)-tensors are defined by
\begin{equation}\label{2.9}
    Q(x,y,z)=g(Q(x,y),z), \quad T(x,y,z)=g(T(x,y),z).
\end{equation}
The symmetry of the Levi-Civita connection implies
\begin{equation}\label{2.10}
    T(x,y)=Q(x,y)-Q(y,x),
\end{equation}
\begin{equation}\label{2.11}
    T(x,y)=-T(y,x).
\end{equation}

\begin{defn}[\cite{Mi}]\label{defn-2.1}
A linear connection $\nn$ on a Riemannian almost product manifold
$(M,P,g)$ is called a \emph{natural connection} if $\nn P=\nn
g=0$.
\end{defn}

If $\nn$ is a linear connection with a tensor $Q$ of the
transformation $\n \rightarrow\nn$ on a Riemannian almost product
manifold, then it is  a natural connection if and only if the
following conditions are valid \cite{Mi}:
\begin{equation*}\label{2.13}
    F(x,y,z)=Q(x,y,Pz)-Q(x,Py,z),
\end{equation*}
\begin{equation}\label{2.14}
    Q(x,y,z)=-Q(x,z,y).
\end{equation}

Let $R'$ be the curvature tensor of a natural connection $\nn$ on
Riemannian almost product manifold $(M,P,g)$. Then, according to
the definitional equalities for $R$ and $R'$, bearing in mind $\nn
g=0$ and equalities \eqref{2.8}, \eqref{2.9}, \eqref{2.10},
\eqref{2.11} and \eqref{2.14}, we have the following relation for
$R$ and $R'$:
\begin{equation}\label{2.15}
\begin{split}
R(x,y,z,w)&=R'(x,y,z,w)-Q(T(x,y),z,w)\\[4pt]
&- \left(\nn_x Q\right)(y,z,w)+\left(\n_yQ\right)(x,z,w)\\[4pt]
&+g\bigl(Q(x,z),Q(y,w)\bigr)-g\bigl(Q(y,z),Q(x,w)\bigr).
\end{split}
\end{equation}


\section{A natural connection $D$ on $(M,P,g)\in\W_1$}\label{sec2}

In the classification of Riemannian almost product manifolds given
in \cite{Sta-Gri}, the class $\W_1$ is the only class, where the
basic tensor $F$ is expressed in terms of the tensor
$g\otimes\ta$, namely \eqref{2.4}.

Let $T$ be the torsion of an arbitrary natural connection on a
Riemannian product manifold $(M,P,g)\in\W_1$. In \cite{Dobr}, it
is obtained the following expression of $T$ by $g\otimes\ta$:
\begin{equation}\label{3.1}
\begin{split}
    &T(x,y,z)=\frac{1}{2n}\left\{g(y,z)\ta(Px)-g(x,z)\ta(Py)\right\}\\[4pt]
            &+\lm\left\{g(y,z)\ta(x)-g(x,z)\ta(y)+g(y,Pz)\ta(Px)-g(x,Pz)\ta(Py)\right\}\\[4pt]
            &+\mu\left\{g(y,Pz)\ta(x)-g(x,Pz)\ta(y)+g(y,z)\ta(Px)-g(x,z)\ta(Py)\right\},
\end{split}
\end{equation}
where $\lm,\mu\in\R$. Let us note that for $\lm=0$ and
$\mu=-\frac{1}{4n}$ we have the torsion of the canonical
connection investigated in \cite{Sta-Gri}. The canonical
connection on an arbitrary Riemannian almost product manifold is
introduced in \cite{Mi} as an analogue of the Hermitian connection
in Hermitian geometry (\cite{Li-2}, \cite{Li-1}, \cite{Ya}).

The goal of the present work is the investigation of the natural
connection $D$ whose torsion $T$ is determined by \eqref{3.1} for
$\lm=\mu=0$, \ie
\begin{equation}\label{3.2}
    T(x,y,z)=\frac{1}{2n}\left\{g(y,z)\ta(Px)-g(x,z)\ta(Py)\right\}.
\end{equation}

\begin{prop}\label{prop-3.1}
    The connection $D$ is determined by the following equality
    \begin{equation}\label{3.3}
    g\left(D_xy,z\right)=g\left(\n_xy,z\right)+Q(x,y,z),
    \end{equation}
    where
    \begin{equation}\label{3.4}
    Q(x,y,z)=\frac{1}{2n}\left\{g(x,y)\ta(Pz)-g(x,z)\ta(Py)\right\}.
    \end{equation}
\end{prop}
\begin{proof}
Since $D$ is a natural connection, according to \cite{Hay}, there
is valid the following equality for the tensor $Q$ of the
transformation $\n\rightarrow D$:
\begin{equation}\label{3.4'}
    Q(x,y,z)=\frac{1}{2n}\left\{T(x,y,z)-T(y,z,x)+T(z,x,y)\right\}.
\end{equation}
Applying \eqref{3.2} to \eqref{3.4'}, we have \eqref{3.3}.
\end{proof}

For the corresponding tensor $Q$ of type (1,2), according to
\eqref{3.4}, we obtain
\begin{equation}\label{3.5}
    Q(x,y)=\frac{1}{2n}\left\{g(x,y)P\Omega-\ta(Py)x\right\},
\end{equation}
where $\Omega$ is determined by $g(\Omega,x)=\ta(x)$. Then
\eqref{3.3} implies
\begin{equation}\label{3.6}
    D_xy=\n_xy+\frac{1}{2n}\left\{g(x,y)P\Omega-\ta(Py)x\right\},
\end{equation}
which yields the following relation between the covariant
derivatives of $\ta$ with respect $D$ and $\n$:
\begin{equation}\label{3.7}
    \left(D_x\ta\right)y=\left(\n_x\ta\right)y
    -\frac{1}{2n}\left\{g(x,y)\ta(P\Omega)-\ta(Py)\ta(x)\right\}.
\end{equation}

We have the following immediately consequence of \eqref{3.5}:
\begin{equation}\label{3.8}
    T(x,y)=\frac{1}{2n}\left\{\ta(Px)y-\ta(Py)\ta(x)\right\}.
\end{equation}

The equality \eqref{3.2} implies the properties
\begin{equation}\label{3.9}
    \mathop{\s}_{x,y,z} T(x,y,z)=\mathop{\s}_{x,y,z}T(Px,Py,z)=0,
\end{equation}
where $\mathop{\s}_{x,y,z}$ is the cyclic sum by $x, y, z$.

Equalities \eqref{3.9} and \eqref{3.4'} yield
\begin{equation*}\label{3.9'}
    Q(x,y,z)=T(z,y,x).
\end{equation*}

According to \eqref{3.8}, it is follows $\ta(PT(x,y))=0$ and then,
bearing in mind \eqref{3.2}, we obtain
\begin{equation*}
    T(T(x,y),z)=\frac{1}{4n^2}\left\{\ta(Py)\ta(Pz)x-\ta(Px)\ta(Pz)y\right\}.
\end{equation*}
Therefore we have the property
\begin{equation}\label{3.10}
    \mathop{\s}_{x,y,z} T(T(x,y),z)=0.
\end{equation}

\begin{rem}
According to property \eqref{3.10}, the case $T(x,y)=-[x,y]$, \ie
$D_xy=D_yx$, is possible (\cite{Hel}, \cite{KamTon}). Then the
manifold has a structure of a Lie group.
\end{rem}

According to \eqref{2.15}, \eqref{3.4}, \eqref{3.5} and
\eqref{3.8}, we obtain the following identity for the curvature
tensors $R$ and $R'$ of the connections $\n$ and $D$:
\begin{equation}\label{3.11}
\begin{split}
R(x,y,z,w)&=R'(x,y,z,w)-\frac{\ta(\Omega)}{4n^2}\pi_1(x,y,z,w)\\[4pt]
&-\frac{1}{2n}\left\{g(y,z)\left(D_x\ta\right)Pw-g(x,z)\left(D_y\ta\right)Pw\right.\\[4pt]
&\phantom{-\frac{1}{2n}\left\{\right.}
\left.+g(x,w)\left(D_y\ta\right)Pz-g(y,w)\left(D_x\ta\right)Pz\right\},
\end{split}
\end{equation}
where
\begin{equation}\label{3.12}
    \pi_1(x,y,z,w)=g(y,z)g(x,w)-g(x,z)g(y,w).
\end{equation}

Since $\mathop{\s}_{x,y,z} R(x,y,z,w)=0$ and $\mathop{\s}_{x,y,z}
\pi_1(x,y,z,w)=0$, then equality \eqref{3.11} implies
\begin{equation}\label{3.12'}
    \mathop{\s}_{x,y,z} R'(x,y,z,w)=\frac{1}{2n}
    \mathop{\s}_{x,y,z}
    g(x,w)\left\{\left(D_y\ta\right)Pz-\left(D_z\ta\right)Py\right\}.
\end{equation}

In \cite{Sta-Gri}, it is introduced the following curvature-like
tensor $\psi_1$ on a Riemannian product manifold
\begin{equation}\label{3.13}
\begin{split}
    \psi_1(S)(x,y,z,w)&=g(y,z)S(x,w)-g(x,z)S(y,w)\\[4pt]
                    &+S(y,z)g(x,w)-S(x,z)g(y,w),
\end{split}
\end{equation}
where
\begin{equation}\label{3.14}
    S(x,y)=\left(D_x\ta\right)Py+\frac{\ta(\Omega)}{2n}g(x,y).
\end{equation}

Using \eqref{3.11}, \eqref{3.12}, \eqref{3.13} and \eqref{3.14},
we obtain the following
\begin{thm}\label{thm-3.2}
The curvature tensors $R$ and $R'$ of $\n$ and $D$ are related via
the formula
\begin{equation}\label{3.15}
    R(x,y,z,w)=R'(x,y,z,w)-\frac{1}{2n}\psi_1(S)(x,y,z,w).
\end{equation}
\qed
\end{thm}

\begin{cor}\label{cor-3.3}
The Ricci tensors $\rho$ and $\rho'$ as well as the scalar
curvatures $\tau$ and $\tau'$ for $\n$ and $D$, respectively, are
related as follows
\begin{gather}
    \rho(y,z)=\rho'(y,z)-\frac{1}{2n}\left\{g(y,z)\tr S+2(n-1)S(y,z)\right\},\label{3.16}\\[4pt]
    \tau=\tau'-\frac{2n-1}{n}\tr S.\label{3.17}
\end{gather}
\qed
\end{cor}


\section{Connection $D$ with Riemannian $P$-tensor of curvature}\label{sec3}

Conditions \eqref{2.5} are satisfies for the curvature tensor $R'$
of the connection $D$ because of \eqref{3.11}. Moreover, condition
\eqref{2.7} is also valid for $R'$, because of $DP=0$. Therefore,
$R'$ is a Riemannian $P$-tensor if and only if $R'$ satisfies
condition \eqref{2.6}.
\begin{thm}\label{thm-4.1}
The connection $D$ has a Riemannian $P$-tensor of curvature if and
only if the following condition is valid
\begin{equation}\label{4.3}
    \left(D_y\ta\right)Pz=\left(D_z\ta\right)Py.
\end{equation}
\end{thm}
\begin{proof}
Let the curvature tensor $R'$ of the connection $D$ be a
Riemannian $P$-tensor. Then \eqref{2.6} is valid for $R'$ and
because of \eqref{3.12'} we have
\begin{equation}\label{4.2}
        \mathop{\s}_{x,y,z}
    g(x,w)\left\{\left(D_y\ta\right)Pz-\left(D_z\ta\right)Py\right\}=0.
\end{equation}
Since we have supposed that $\dim M\geq 4$, then \eqref{4.2}
implies \eqref{4.3}.

Vice versa, if \eqref{4.3} is satisfied, then \eqref{3.12'} yields
condition \eqref{2.6} for $R'$ and therefore $R'$ is a Riemannian
$P$-tensor.
\end{proof}

Bearing in mind \eqref{4.3} and \eqref{3.7}, we have immediately
the following
\begin{cor}\label{cor-4.2}
The connection $D$ has a Riemannian $P$-tensor of curvature if and
only if the following condition is valid
\begin{equation*}
    \left(\n_x\ta\right)Py=\left(\n_y\ta\right)Px,
\end{equation*}
\ie if and only if the 1-form $\ta\circ P$ is closed. \qed
\end{cor}


\section{Connection $D$ with parallel torsion}\label{sec4}

Equality \eqref{3.8} implies
\begin{equation*}
        \left(D_xT\right)(y,z)=\frac{1}{2n}\left\{\left(D_x\ta\right)Py.z-\left(D_x\ta\right)Pz.y\right\},
\end{equation*}
which yields immediately the following
\begin{prop}\label{prop-5.1}
The connection $D$ has parallel torsion if and only if the
associated Lee 1-form $\ta$ is also parallel with respect to $D$.
\qed
\end{prop}

Because of \eqref{3.7}, it is valid the following
\begin{cor}\label{cor-5.2}
The connection $D$ has parallel torsion if and only if the
following condition is satisfied
\begin{equation*}
        \left(\n_x\ta\right)y=\frac{1}{2n}\left\{g(x,y)\ta(P\Omega)-\ta(Py)\ta(x)\right\}.
\end{equation*}
\qed
\end{cor}

Let $D$ have a parallel torsion. Then, according to \eqref{3.11}
we obtain
\begin{equation}\label{5.1}
    R(x,y,z,w)=R'(x,y,z,w)-\frac{\ta(\Omega}{4n^2}\pi_1(x,y,z,w),
\end{equation}
which implies condition \eqref{2.6} for $R'$. Therefore, it is
valid the following
\begin{prop}\label{prop-5.3}
If the connection $D$ has parallel torsion then the curvature
tensor of $D$ is a Riemannian $P$-tensor which satisfies condition
\eqref{5.1}. \qed
\end{prop}


\section{The Weyl tensor of the transformation $\n\rightarrow D$}\label{sec5}

Let $W$ and $W'$ be the Weyl tensors for the connections $\n$ and
$D$, respectively, \ie
\begin{equation}\label{6.1}
\begin{split}
    &W=R-\frac{1}{2(n-1)}\left\{\psi_1(\rho)-\frac{\tau}{2n-1}\pi_1\right\},\\[4pt]
    &W'=R'-\frac{1}{2(n-1)}\left\{\psi_1(\rho')-\frac{\tau'}{2n-1}\pi_1\right\}.
\end{split}
\end{equation}

\begin{thm}\label{thm-6.1}
The Weyl tensor is invariant with respect to the transformation
$\n\rightarrow D$.
\end{thm}
\begin{proof}
We determine $\tr S$ from \eqref{3.17} and replace it in
\eqref{3.16}. So we obtain
\begin{equation*}
    \frac{n-1}{n}S=\left\{\rho'-\frac{\tau'}{2(2n-1)}g\right\}-\left\{\rho-\frac{\tau}{2(2n-1)}g\right\}.
\end{equation*}
After that we determine $S$ from the latter equality and replace
it in \eqref{3.15}. In the result we take into account
\eqref{3.13} and the equality $\psi_1(g)=2\pi_1$ and obtain $W=W'$
by appropriate regrouping.
\end{proof}


\section{Conformal transformation of the curvature tensor of $D$}\label{sec6}

Now we will establish the way of transforming of the curvature
tensor $R'$ of the connection $D$ using the usual conformal
transformation of the metric $g$. This transformation is
determined via the formula
\begin{equation}\label{7.1}
    \bar{g}=e^{2u}g,
\end{equation}
where $u$ is a smooth function on the considered manifold.

It is known that the Levi-Civita connection is transformed by
\eqref{7.1} as follows
\begin{equation}\label{7.2}
    \bar{\n}_xy=\n_xy+\dd u(x)y+\dd u(y)x-g(x,y)L,
\end{equation}
where $L=\grad u$.

In \cite{Sta-Gri}, it is proved that $(M,P,g)\in\W_1$ implies
$(M,P,\bar{g})\in\W_1$, such as the associated Lee 1-forms $\ta$
and $\bar{\ta}$ as well as their corresponding vectors $\Omega$
and $\bar{\Omega}$ are related via the formulas
\begin{gather}
    \bar{\ta}(x)=\ta(x)+2u\dd u(Px),\label{7.3}\\[4pt]
    \bar{\Omega}=e^{-2u}\left(\Omega+2nL\right).\label{7.4}
\end{gather}

Analogously to \eqref{3.6} it is valid the following
\begin{equation*}
    \bar{D}_xy=\bar{\n}_xy+\frac{1}{2n}\left\{\bar{g}(x,y)P\bar{\Omega}-\bar{\ta}(Py)x\right\},
\end{equation*}
The latter equality and equalities \eqref{7.1}, \eqref{7.2},
\eqref{7.3} and \eqref{7.4} yield
\begin{equation}\label{7.5}
    \bar{D}_xy=D_xy+\dd u(x)y.
\end{equation}
Using the definitional equality
$\bar{R}'(x,y)z=\bar{D}_x\bar{D}_yz-\bar{D}_y\bar{D}_xz-\bar{D}_{[x,y]}z$
for the curvature tensor $\bar{R}'$ of $\bar{D}$ and \eqref{7.5},
we obtain $\bar{R}'=R'$, \ie it is valid the following
\begin{thm}\label{thm-7.1}
The curvature tensor of the connection $D$ is invariant with
respect to the usual conformal transformation \eqref{7.1} of the
metric $g$. \qed
\end{thm}


\section{The case of a flat connection $D$}\label{sec7}

Let $D$ be a flat connection, \ie $R'=0$. Then $W'=0$ and because
of \thmref{thm-6.1} we have $W=0$. Therefore it is valid the
following
\begin{prop}\label{prop-8.1}
If $D$ is a flat connection then the manifold is conformally flat
with respect to $\n$. \qed
\end{prop}

Further, we have the following
\begin{thm}\label{thm-8.2}
Let the connection $D$ be a flat connection with parallel torsion.
Then the following propositions are valid:
\begin{enumerate}    \renewcommand{\labelenumi}{(\roman{enumi})}
    \item
    $R=-\frac{1}{4n^2}\ta(\Omega)\pi_1$, $\quad\rho=-\frac{2n-1}{4n^2}\ta(\Omega)g$,
    $\quad\tau=-\frac{2n-1}{2n}\ta(\Omega)$;
    \item The tensor $R$ is parallel with respect to $D$;
    \item The manifold is a space form;
    \item The scalar curvature $\tau$ is negative.
\end{enumerate}
\end{thm}
\begin{proof}
Let the connection $D$ be a flat connection with parallel torsion.
Then, according to the \propref{prop-5.3}, we obtain the first
equality in (i), which implies the other two equalities of (i).

Bearing in mind \eqref{3.3}, we have
\begin{equation}\label{8.2}
\begin{split}
&R(x,y)z=R'(x,y)z -\left(D_x Q\right)(y,z)+\left(D_y Q\right)(x,z)
\\[4pt]
&\phantom{R(x,y)z=R'(x,y)z}
+Q\bigl(x,Q(y,z)\bigr)-Q\bigl(y,Q(x,z)\bigr).
\end{split}
\end{equation}
Since $DT=0$, using \eqref{3.4'}, we obtain $DQ=0$. Then relation
\eqref{8.2} implies $DR=DR'$. Hence, because of $DR'=0$, we obtain
$DR=0$, \ie (ii) is valid.

The formulas in (i) imply directly (iii).

Since $g$ is a Riemannian metric, then
$\ta(\Omega)=g(\Omega,\Omega)> 0$. Therefore, because of the
latter equality in (i), we have that $\tau<0$, \ie (iv) is valid.
\end{proof}


\section{Example}\label{sec8}

\subsection{A Lie group $G$ as a Riemannian product $\W_1$-manifold $(G,P,g)$}

Let $G$ be a 4-dimensional real connected Lie group and $\g$ be
its Lie algebra with a basis $\{X_i\}$.

We introduce a structure $P$ and left invariant metric $g$ as
follows
\begin{equation}\label{9.1}
    PX_1=X_3,\quad PX_2=X_4,\quad PX_3=X_{1},\quad PX_4=X_{2},
\end{equation}
\begin{equation}\label{9.2}
    g(X_i,X_j)=
\begin{cases}
\begin{array}{rl}
1, \quad & i=j;\\
0, \quad & i\neq j.
\end{array}
\end{cases}
\end{equation}
Thus, $(G,P,g)$ becomes a Riemannian almost product manifold with
$\tr{P}=0$.

We will consider the case when $P$ is an Abelian almost product
structure \cite{AnBaDoOv}, \ie
\begin{equation}\label{9.3}
    [PX_i,PX_j]=-[X_i,X_j].
\end{equation}
Then the manifold $(G,P,g)$ has a zero Nijenhuis tensor and
therefore $(G,P,g)$ belongs to the class $\W_1\oplus\W_2$.

\begin{thm}\label{thm-9.1}
The manifold $(G,P,g)$ is a Riemannian product manifold belonging
to the class $\W_1$ if and only if the Lie algebra $\g$ is
determined by the conditions:
\begin{equation}\label{9.5}
\begin{array}{l}
[X_1,X_2]=-[X_3,X_4]=\lm_1 X_1+\lm_2 X_2+\lm_3 X_3+\lm_{4} X_4,\\[4pt]
[X_1,X_3]=[X_2,X_4]=\lm_4 X_1-\lm_3 X_2+\lm_2 X_3-\lm_1 X_4,\\[4pt]
[X_2,X_3]=[X_1,X_4]=0,\qquad (\lm_i\in \R;\; i=1,2,3,4).
\end{array}
\end{equation}
\end{thm}
\begin{proof}
Because of \eqref{9.2} we have
\begin{equation}\label{9.3'}
\begin{split}
    2g\bigl(\n_{X_i} X_j,X_k\bigr)
    =
    g\left([X_i,X_j],X_k\right)&+g\left([X_k,X_i],X_j\right)\\[4pt]
    &+g\left([X_k,X_j],X_i\right),
\end{split}
\end{equation}
which implies the following equality, according to \eqref{9.3} and
\eqref{2.2}:
\begin{equation}\label{9.4}
\begin{split}
    2F\left(X_i,X_j,X_k\right)&=g\bigl([X_i,PX_j]-P[X_i,X_j],X_k\bigr)\\[4pt]
    &+g\bigl([X_i,PX_k]-P[X_i,X_k],X_j\bigr)\\[4pt]
    &+2g\bigl([X_k,PX_j],X_i\bigr).
\end{split}
\end{equation}
The comparing of condition \eqref{9.4} with the characteristic
condition \eqref{2.4} for the class $\W_1$, taking into account
\eqref{9.1}, \eqref{9.2}, \eqref{9.3} and the Jacobi identity for
the commutators $[X_i,X_j]$, yields conditions \eqref{9.5}.
\end{proof}

The comparing of \eqref{9.4} and \eqref{2.4} give also the
following formulas for the associated Lee 1-form $\ta$:
\begin{equation}\label{9.6}
\ta_1=4\lm_4,\quad \ta_2=-4\lm_3,\quad \ta_3=-4\lm_2,\quad
\ta_4=4\lm_1.
\end{equation}

Further, $(G,P,g)$ will stand for the manifold determined by
\eqref{9.5}.

\subsection{Some geometrical characteristics of the manifold $(G,P,g)$}

By virtue of \eqref{9.3'} and \eqref{9.5} we get the non-zero
components $\n_{X_i}X_j$ of the Levi-Civita connection $\n$:
\begin{equation}\label{9.7}
    \begin{array}{l}
\n_{X_1} X_1=\n_{X_4} X_4=-\lm_1 X_2-\lm_4 X_3,\\[4pt]%
\n_{X_2} X_2=\n_{X_3} X_3=\lm_2 X_1+\lm_3 X_4,\\[4pt]%
\n_{X_1} X_2=-\n_{X_3} X_4=\lm_1 X_1+\lm_3 X_3,\\[4pt]%
\n_{X_2} X_1=-\n_{X_4} X_3=-\lm_2 X_2-\lm_4 X_4,\\[4pt]%
\n_{X_1} X_3=\n_{X_2} X_4=\lm_4 X_1-\lm_3 X_2,\\[4pt]%
\n_{X_3} X_1=\n_{X_4} X_2=-\lm_2 X_3+\lm_1 X_4.
    \end{array}
\end{equation}

Using \eqref{9.7} and \eqref{9.5} we obtain the non-zero
components $R_{ijks}=R(X_i,\allowbreak X_j,X_k,\allowbreak{}X_s)$
of the curvature tensor $R$ for $\n$:
\begin{equation}\label{9.8}
\begin{array}{c}
\begin{array}{ll}
    R_{1212}=\lm_1^2+\lm_2^2,\quad & R_{1313}=\lm_2^2+\lm_4^2,\\[4pt]
    R_{1414}=\lm_1^2+\lm_4^2,\quad & R_{2323}=\lm_2^2+\lm_3^2,\\[4pt]
    R_{2424}=\lm_1^2+\lm_3^2,\quad & R_{3434}=\lm_3^2+\lm_4^2,\\[4pt]
\end{array}\\[4pt]
\begin{array}{ll}
    R_{1213}=R_{2134}=\lm_1\lm_4,\quad & R_{1214}=R_{2343}=\lm_2\lm_4,\\[4pt]
    R_{1232}=R_{1434}=\lm_1\lm_3,\quad & R_{1242}=R_{1343}=\lm_2\lm_3,\\[4pt]
    R_{1341}=R_{2342}=\lm_1\lm_2,\quad & R_{1332}=R_{1442}=\lm_3\lm_4.
\end{array}
\end{array}
\end{equation}
The rest of the non-zero components are obtained by the properties
\[
R_{ijks}=R_{ksij},\qquad R_{ijks}=-R_{jiks}=-R_{ijsk}.
\]

Using \eqref{9.8} for the non-zero components
$\rho_{ij}=\rho(X_i,X_j)$ of the Ricci tensor $\rho$ we compute:
\begin{equation}\label{9.9}
\begin{array}{c}
\begin{array}{ll}
    \rho_{11}=-2\left(\lm_1^2+\lm_2^2+\lm_4^2\right),\quad
    &
    \rho_{22}=-2\left(\lm_1^2+\lm_2^2+\lm_3^2\right),\\[4pt]
    \rho_{33}=-2\left(\lm_2^2+\lm_3^2+\lm_4^2\right),\quad
    &
    \rho_{44}=-2\left(\lm_1^2+\lm_3^2+\lm_4^2\right),\\[4pt]
\end{array}\\[4pt]
\begin{array}{lll}
    \rho_{12}=2\lm_3\lm_4,\quad & \rho_{13}=-2\lm_1\lm_3,\quad & \rho_{14}=-2\lm_2\lm_3,\\[4pt]
    \rho_{34}=2\lm_1\lm_2,\quad & \rho_{23}=-2\lm_1\lm_4,\quad & \rho_{24}=-2\lm_2\lm_4.
\end{array}
\end{array}
\end{equation}
The rest of the non-zero components are obtained by the property
$\rho_{ij}=\rho_{ji}$.

By \eqref{9.9} we obtain the following
\begin{prop}\label{prop-8.1'}
The manifold $(G,P,g)$ has a negative scalar curvature for $\n$:
\begin{equation}\label{9.10}
    \tau=-6\left(\lm_1^2+\lm_2^2+\lm_3^2+\lm_4^2\right).
\end{equation}
\qed
\end{prop}

For the Riemannian sectional curvatures of the $P$-invariant basis
2-planes \allowbreak $(X_1,X_3)$ and $(X_2,X_4)$, i.e. for the
invariant sectional curvatures of the basis 2-planes, we get
\begin{equation}\label{9.11}
    k_{13}=-\left(\lm_2^2+\lm_4^2\right),\qquad
    k_{24}=-\left(\lm_1^2+\lm_3^2\right).
\end{equation}
The sectional curvatures of the rest of the basis 2-planes, i.e.
the anti-invariant sectional curvatures of the basis 2-planes,
are:
\begin{equation}\label{9.12}
    \begin{array}{ll}
    k_{12}=-\left(\lm_1^2+\lm_2^2\right),\qquad &
    k_{14}=-\left(\lm_1^2+\lm_4^2\right),\\[4pt]
    k_{23}=-\left(\lm_2^2+\lm_3^2\right),\qquad &
    k_{34}=-\left(\lm_3^2+\lm_4^2\right).
\end{array}
\end{equation}

Conditions \eqref{9.11} and \eqref{9.12} imply the following
\begin{thm}\label{thm-9.2}
The manifold $(G,P,g)$ has:
\begin{enumerate}\renewcommand{\labelenumi}{(\roman{enumi})}
    \item
    a constant invariant sectional curvature if and only if
\[
\lm_1^2-\lm_2^2+\lm_3^2-\lm_4^2=0;
\]
    \item
    a constant anti-invariant sectional curvature if and only if
\[
\lm_1^2=\lm_3^2,\qquad \lm_2^2=\lm_4^2;
\]
    \item
    a constant sectional curvature if and only if
\[
\lm_1^2=\lm_2^2=\lm_3^2=\lm_4^2.
\]
In this case
$R=\frac{\tau}{12}\pi_1$. \qed
\end{enumerate}
\end{thm}

By virtue of \eqref{6.1}, using \eqref{9.8}, \eqref{9.9},
\eqref{9.10}, \eqref{3.12} and \eqref{3.13} for $S=\rho$, we
obtain the following
\begin{prop}\label{prop-8.2'}
The manifold $(G,P,g)$ has a zero Weyl tensor, \ie $(G,P,g)$ is
conformally flat for $\n$. \qed
\end{prop}

\subsection{The connection $D$ on the manifold $(G,P,g)$}

Further in our considerations we exclude the trivial case
$\lm_1=\lm_2=\lm_3=\lm_4=0$, \ie the case when $(G,P,g)$ is a
Riemannian $P$-manifold.

By virtue of \eqref{3.6}, \eqref{9.7} and \eqref{9.2} for the
components of the connection $D$ on  $(G,P,g)$ we obtain:
\begin{equation}\label{9.13}
    \begin{array}{ll}
D_{X_1} X_1=-\lm_3 X_4,\quad & D_{X_1} X_2=\lm_3 X_3,\\[4pt]%
D_{X_1} X_3=-\lm_3 X_2,\quad & D_{X_1} X_4=\lm_3 X_1,\\[4pt]%
D_{X_2} X_1=-\lm_4 X_4,\quad & D_{X_2} X_2=\lm_4 X_3,\\[4pt]%
D_{X_2} X_3=-\lm_4 X_2,\quad & D_{X_2} X_4=\lm_4 X_1,\\[4pt]%
D_{X_3} X_2=-\lm_1 X_3,\quad & D_{X_3} X_1=\lm_1 X_4,\\[4pt]%
D_{X_3} X_4=-\lm_1 X_1,\quad & D_{X_3} X_3=\lm_1 X_2,\\[4pt]%
D_{X_4} X_2=-\lm_2 X_3,\quad & D_{X_4} X_1=\lm_2 X_4,\\[4pt]%
D_{X_4} X_4=-\lm_2 X_1,\quad & D_{X_4} X_3=\lm_2 X_2.
    \end{array}
\end{equation}

Using \eqref{9.13} and \eqref{9.5}, we obtain that all components
of the curvature tensor $R'$ of $D$ are zeros, \ie the following
proposition holds:
\begin{prop}\label{prop-8.2''}
The manifold $(G,P,g)$ has a flat connection $D$. \qed
\end{prop}

By virtue of \eqref{3.8} and \eqref{9.6} we get the non-zero
components $T_{ij}=T(X_i,X_j)$ of the torsion $T$ for the
connection $D$:
\begin{equation*}\label{9.14}
    \begin{array}{ll}
T_{12}=-\lm_1 X_1-\lm_2 X_2,\quad & T_{13}=-\lm_4 X_1-\lm_2 X_3,\\[4pt]%
T_{14}=\lm_3 X_1-\lm_2 X_4,\quad & T_{23}=-\lm_4 X_2+\lm_1 X_3,\\[4pt]%
T_{24}=\lm_3 X_2+\lm_1 X_4,\quad & T_{34}=\lm_3 X_3+\lm_4 X_4.
    \end{array}
\end{equation*}
The rest of the non-zero components are obtained by the property
$T_{ij}=-T_{ji}$.

According to \propref{prop-5.1}, the connection $D$ has a parallel
torsion if and only if $\left(D_{X_i}\ta\right)X_j=0$. The latter
equality is equivalent to $\lm_1=\lm_2=\lm_3=\lm_4=0$, because of
\eqref{9.6} and \eqref{9.13}. This case is excluded from our
investigations. Therefore, the following proposition holds:
\begin{prop}\label{prop-8.3}
The connection $D$ on the manifold $(G,P,g)$ is not parallel. \qed
\end{prop}


\bigskip

\textit{\\
Dobrinka Gribacheva\\
Department of Geometry\\
Faculty of Mathematics and Informatics
\\
Paisii Hilendarski University of Plovdiv\\
236 Bulgaria Blvd.\\
4003 Plovdiv, Bulgaria
\\
e-mail: dobrinka@uni-plovdiv.bg}

\end{document}